\documentclass{amsart}

\usepackage[all, cmtip]{xy}
\usepackage{amsmath,amsfonts,amssymb,amsthm,mathtools,mathrsfs}
\usepackage{subcaption,tikz}
\usepackage[T1]{fontenc}

\theoremstyle{definition}
\newtheorem{defn}{Definition}[section]

\newtheorem{example}[defn]{Example}
\newtheorem{remark}[defn]{Remark}

\theoremstyle{plain}
\newtheorem{lemma}[defn]{Lemma}
\newtheorem{theorem}[defn]{Theorem}
\newtheorem{proposition}[defn]{Proposition}
\newtheorem{corollary}[defn]{Corollary}

\newcommand{\Z}{\mathbb{Z}}
\newcommand{\Q}{\mathbb{Q}}
\newcommand{\C}{\mathbb{C}}
\newcommand{\Proj}{\mathbb{P}}
\newcommand{\calO}{\mathcal{O}}
\newcommand{\calM}{\mathcal{M}}
\newcommand{\calX}{\mathcal{X}}
\newcommand{\calY}{\mathcal{Y}}
\newcommand{\calL}{\mathcal{L}}
\newcommand{\calS}{\mathcal{S}}

\DeclareMathOperator{\NS}{NS}
\DeclareMathOperator{\Pic}{Pic}
\DeclareMathOperator{\rank}{rank}
\DeclareMathOperator{\Gr}{Gr}
\DeclareMathOperator{\GL}{GL}
\DeclareMathOperator{\hcf}{hcf}
\DeclareMathOperator{\Orth}{O}

\begin{document}

\title[Calabi-Yau Threefolds Fibred by Mirror Quartic K3 Surfaces]{Calabi-Yau Threefolds Fibred by Mirror Quartic K3 Surfaces}

\author[C. F. Doran]{Charles F. Doran}
\address{Department of Mathematical and Statistical Sciences, 632 CAB, University of Alberta, Edmonton, AB, T6G 2G1, Canada}
\email{charles.doran@ualberta.ca}
\thanks{C. F. Doran and A. Y. Novoseltsev were supported by the Natural Sciences and Engineering Resource Council of Canada (NSERC), the Pacific Institute for the Mathematical Sciences (PIMS), and the McCalla Professorship at the University of Alberta}

\author[A. Harder]{Andrew Harder}
\address{Department of Mathematical and Statistical Sciences, 632 CAB, University of Alberta, Edmonton, AB, T6G 2G1, Canada}
\email{aharder@ualberta.ca}
\thanks{A. Harder was supported by an NSERC Post-Graduate Scholarship}

\author[A.Y. Novoseltsev]{Andrey Y. Novoseltsev}
\address{Department of Mathematical and Statistical Sciences, 632 CAB, University of Alberta, Edmonton, AB, T6G 2G1, Canada}
\email{novoselt@ualberta.ca}

\author[A. Thompson]{Alan Thompson}
\address{Department of Pure Mathematics, University of Waterloo, 200 University Ave West, Waterloo, ON, N2L 3G1, Canada}
\email{am6thomp@uwaterloo.ca}
\thanks{A. Thompson was supported by a Fields-Ontario-PIMS Postdoctoral Fellowship with funding provided by NSERC, the Ontario Ministry of Training, Colleges and Universities, and an Alberta Advanced Education and Technology Grant}
\date{}

\begin{abstract} We study threefolds fibred by mirror quartic K3 surfaces. We begin by showing that any family of such K3 surfaces is completely determined by a map from the base of the family to the moduli space of mirror quartic K3 surfaces. This is then used to give a complete explicit description of all Calabi-Yau threefolds fibred by mirror quartic K3 surfaces. We conclude by studying the properties of such Calabi-Yau threefolds, including their Hodge numbers and deformation theory.
\end{abstract}

\maketitle

\section{Introduction}

In this paper, along with its predecessor \cite{flpk3sm}, we take the first steps towards a systematic study of threefolds fibred by K3 surfaces, with a particular focus on Calabi-Yau threefolds. Our aim in this paper is to gain a complete understanding of a relatively simple case, where the generic fibre in the K3 fibration is a mirror quartic, to demonstrate the utility of our methods and to act as a test-bed for developing a more general theory.

We have chosen mirror quartic K3 surfaces (here ``mirror'' is used in the sense of Nikulin \cite{fagkk3s} and Dolgachev \cite{mslpk3s}) because, from a moduli theoretic perspective, they may be thought of as the simplest non-rigid lattice polarized K3 surfaces. Indeed, mirror quartic K3 surfaces are polarized by the rank $19$ lattice
\[M_2 := H \oplus E_8 \oplus E_8 \oplus \langle -4 \rangle,\]
so move in a $1$-dimensional moduli space. By \cite[Theorem 7.1]{mslpk3s}, this moduli space is isomorphic to the modular curve $\Gamma_0(2)^+ \setminus \mathbb{H}$, we denote its compactification by $\calM_{M_2}$.

Our first main result (Theorem \ref{gfithm}) will show that an $M_2$-polarized family of K3 surfaces (in the sense of \cite[Definition 2.1]{flpk3sm}) over a quasi-projective base curve $U$ is completely determined by its \emph{generalized functional invariant map} $U \to \calM_{M_2}$, which may be thought of as a K3 analogue of the classical functional invariant of an elliptic curve. This also explains why we choose to polarize our K3 surfaces by $M_2$ instead of $M_1 := H \oplus E_8 \oplus E_8 \oplus \langle -2 \rangle$, which at first would seem like a more obvious choice. Indeed, $M_1$-polarized K3 surfaces admit an antisymplectic involution that fixes the polarization, which means that the analogue of Theorem \ref{gfithm} does not hold for them; in analogy with elliptic curves again, the presence of an antisymplectic involution that fixes the polarization means that to uniquely determine an $M_1$-polarized family of K3 surfaces we would also need a \emph{generalized homological invariant}, to control monodromy around singular fibres, whereas for $M_2$-polarized familes the lack of such automorphisms means that the generalized functional invariant alone suffices.

A second reason for choosing mirror quartic K3 surfaces is the fact that the mirror quintic Calabi-Yau threefold admits a fibration by mirror quartics \cite[Theorem 5.10]{flpk3sm}. This makes fibrations by mirror quartic K3 surfaces particularly interesting for the study of Calabi-Yau threefolds; the majority of this paper is devoted to this study. Indeed, our second main result (Corollary \ref{XgCY3}) provides a complete explicit description of all Calabi-Yau threefolds that admit $M_2$-polarized K3 fibrations, and we compute Hodge numbers in all cases. Throughout this study we present the mirror quintic as a running example, thereby demonstrating that many of its known properties can be easily recovered from our theory, although we would like to note that our methods apply to a significantly broader class of examples of Calabi-Yau threefolds, many of which do not have known descriptions as complete intersections in toric varieties. Finally, we note that mirror symmetry for the Calabi-Yau threefolds constructed here is explored in \cite{mstdfcym}.
\medskip

The structure of this paper is as follows. In Section \ref{constructionsect} we begin by proving Theorem \ref{gfithm}, which shows that any $M_2$-polarized family of K3 surfaces is uniquely determined by its generalized functional invariant. In particular, this means that any $M_2$-polarized family of K3 surfaces is isomorphic to the pull-back of a fundamental family of $M_2$-polarized K3 surfaces, introduced in Section \ref{SIsect}, from the moduli space $\calM_{M_2}$. The remainder of Section \ref{constructionsect} is then devoted to showing how this theory can be used to construct Calabi-Yau threefolds, culminating in Corollary \ref{XgCY3}, which gives an explicit description of all Calabi-Yau threefolds that admit $M_2$-polarized K3 fibrations.

In Section \ref{Hodgesect} we begin our study of the properties of the Calabi-Yau threefolds constructed in Section \ref{constructionsect}, by computing their Hodge numbers. The main results are Proposition \ref{Xgh11}, which computes $h^{1,1}$, and Corollary \ref{h21cor}, which computes $h^{2,1}$.

Section \ref{defsect} is devoted to a brief study of the deformation theory of the Calabi-Yau threefolds constructed in Section \ref{constructionsect}. The main result is Proposition \ref{prop:deffib}, which shows that any small deformation of such a Calabi-Yau threefold is induced by a deformation of the generalized functional invariant map of the K3 fibration on it. In particular, this allows us to relate the moduli spaces of such Calabi-Yau threefolds to Hurwitz spaces describing ramified covers between curves, and gives an easy way to study their degenerations.

\section{Construction} \label{constructionsect}

We begin by setting up some notation. Let $\calX$ be a smooth projective threefold that admits a fibration $\pi\colon \calX \to B$ by K3 surfaces over a smooth base curve $B$. Let $\NS(X_p)$ denote the N\'{e}ron-Severi group of the fibre of $\calX$ over a general point $p \in B$. In what follows, we will assume that $\NS(X_p) \cong M_2$, where $M_2$ denotes the rank $19$ lattice $M_2 := H \oplus E_8 \oplus E_8 \oplus \langle -4 \rangle$.

Denote the open set over which the fibres of $\calX$ are smooth K3 surfaces by $U \subset B$ and let $\pi_U\colon \calX_U \to U$ denote the restriction of $\calX$ to $U$. We suppose further that $\calX_U \to U$ is an $M_2$-polarized family of K3 surfaces, in the sense of the following definition.

\begin{defn}\label{defn:L-pol} \textup{\cite[Definition 2.1]{flpk3sm}}
Let $L$ be a lattice and $\pi_U\colon \mathcal{X}_U \rightarrow U$ be a smooth projective family of K3 surfaces over a smooth quasiprojective base $U$. We say that $\mathcal{X}_U$ is an \emph{$L$-polarized family of K3 surfaces} if there is a trivial local subsystem $\mathcal{L}$ of $R^2 \pi_*\mathbb{Z}$ so that, for each $p \in U$, the fibre $\mathcal{L}_p \subseteq H^2({X}_p,\mathbb{Z})$ of $\calL$ over $p$ is a primitive sublattice of $\NS(X_p)$ that is isomorphic to $L$ and contains an ample divisor class.
\end{defn}

To any such family, we can associate a \emph{generalized functional invariant map} $g\colon U \to \calM_{M_2}$, where $\calM_{M_2}$ denotes the (compact) moduli space of $M_2$-polarized K3 surfaces. $g$ is defined to be the map which takes a point $p \in U$ to the point in moduli corresponding to the fibre $X_p$ of $\calX$ over $p$.

\cite[Theorem 5.10]{flpk3sm} gives five examples of Calabi-Yau threefolds admitting such fibrations, arising from the Doran-Morgan classification \cite[Table 1]{doranmorgan}. In each of these cases, the family $\pi_U\colon \calX_U \to U$ is the pull-back of a special family of K3 surfaces $\calX_2 \to \calM_{M_2}$, by the generalized functional invariant map. 

\begin{remark} In addition to the five examples from \cite[Theorem 5.10]{flpk3sm}, the authors are aware of many more Calabi-Yau threefolds which admit such fibrations. Indeed,  the toric geometry functionality of the computer software \emph{Sage} may be used to perform a search for such fibrations on complete intersection Calabi-Yau threefolds in toric varieties that have small Hodge number $h^{2,1}$, yielding dozens of additional examples; details will appear in future work.
\end{remark}

Our first result will show that this is not a coincidence: in fact, any $M_2$-polarized family of K3 surfaces $\pi_U\colon \calX_U \to U$ is determined up to isomorphism by its generalized functional invariant, so we can obtain any such family by pulling back the special family $\calX_2$. We will therefore begin our study of threefolds fibred by $M_2$-polarized K3 surfaces by studying the family $\calX_2$.

\begin{theorem} \label{gfithm} Let $\calX_U \to U$ denote a non-isotrivial $M_2$-polarized family of K3 surfaces over a quasi-projective curve $U$, such that the N\'{e}ron-Severi group of a general fibre of $\calX_U$ is isometric to $M_2$. Then $\calX_U$ is uniquely determined \textup{(}up to isomorphism\textup{)} by its generalized functional invariant map $g\colon U \to \calM_{M_2}$. 
\end{theorem}
\begin{proof} Suppose for a contradiction that $\calX_U$ and $\calY_U$ are two non-isomorphic $M_2$-polarized familes of K3 surfaces over $U$, that satisfy the conditions of the theorem and have the same generalized functional invariant $g\colon U \to \calM_{M_2}$.

Let $\{U_i\}$ denote a cover of $U$ by simply connected open subsets and let $\calX_{U_i}$ (resp. $\calY_{U_i}$) denote the restriction of $\calX_U$ (resp. $\calY_U$) to $U_i$ for each $i$. On each $U_i$, Ehresmann's Theorem (see, for example, \cite[Section 9.1.1]{htcagI}) shows that we can choose markings compatible with the $M_2$-polarizations on the families of K3 surfaces $\calX_{U_i}$ and $\calY_{U_i}$. Thus, by the Global Torelli Theorem \cite[Theorem 2.7']{fagkk3s}, the families $\calX_{U_i}$ and $\calY_{U_i}$ are isomorphic for each $i$.

Therefore, since we have assumed that $\calX_U$ and $\calY_U$ are non-isomorphic, they must differ in how the families $\calX_{U_i}$ and $\calY_{U_i}$ glue together over the intersections $U_i \cap U_j$. Let $V \subset U_i \cap U_j$ be a connected component of such an intersection, such that the gluing maps differ over $V$. As $\calX_U$ and $\calY_U$ are isomorphic over $V$, the gluing maps over $V$ must differ by composition with a nontrivial fibrewise automorphism $\psi$. Moreover, by the polarization condition, $\psi$ must preserve the $M_2$-polarizations on the fibres over $V$.

Now consider the action of $\psi$ on the fibre $X_p$ of $\calX_U$ over a point $p \in V$. Since $\calX_U$ is not isotrivial, we may choose $p$ so that the N\'{e}ron-Severi lattice of $X_p$ is isometric to $M_2$ and so, as $\calX$ is an $M_2$-polarized family of K3 surfaces, $\psi^*\in \Orth (H^2(X_p,\Z))$ fixes $\Pic(X_p) \cong M_2$.  Now, according to \cite[Section 3.3]{fagkk3s}, there is an exact sequence
\[0 \longrightarrow H_{X_p} \longrightarrow \mathrm{Aut}(X_p) \longrightarrow \Orth(M_2),\]
where $H_{X_p}$ is a finite cyclic group of automorphisms of $X_p$ which act nontrivially on $H^{2,0}(X_p)$. Since $\psi \in \mathrm{Aut}(X_p)$ maps to the trivial element of $\Orth(M_2)$, we thus see that $\psi$ must act non-symplectically on $X_p$.

Thus, by \cite[Proposition 1.6.1]{isbfa}, $\psi^*$ descends to an element of the subgroup $\Orth(M_2^\perp)^*$ of  $\Orth(M_2^\perp)$ which induces the trivial automorphism on the discriminant group $A_{M_2^{\perp}}$. Furthermore, since $X_p$ is general, $M_2^\perp$ supports an irreducible rational Hodge structure, so $\psi^*$ must act irreducibly on $M_2^\perp$. Therefore, by \cite[Theorem 3.1]{fagkk3s}, it follows that the order $n$ of $\psi^*$ must satisfy $\varphi(n) | \rank(M_2^\perp) = 3$, where $\varphi(n)$ denotes Euler's totient function. Using Vaidya's \cite{ietf} lower bound  for $\varphi(n)$,
\[\varphi(n) \geq \sqrt{n}\  \mathrm{for}\  n > 2,\  n \neq 6\]
we see that $\varphi(n) |3$ implies that $n \leq 9$. A simple check then shows that $n = 2$ or $n=1$. If $n=2$ then, by irreducibility, $\psi^*$ would have to act as $-\mathrm{Id}$ on $M_2^{\perp}$ and as the identity on $M_2$. However, since the discriminant group of $M_2^\perp$ is $A_{M_2^{\perp}} \cong \mathbb{Z}/4\mathbb{Z}$, such $\psi^*$ would not descend to the identity on  $A_{M_2^{\perp}}$, so this case cannot occur. Therefore, $\psi^*$ must be of order $1$. But then $\psi$ must be the trivial automorphism, which is a contradiction.
\end{proof}

\subsection{A fundamental family}\label{SIsect}

The family $\calX_2 \to \calM_{M_2}$ is described in \cite[Section 5.4.1]{flpk3sm}. It is given as the minimal resolution of the family of hypersurfaces in $\Proj^3$ obtained by varying $\lambda$ in the following expression
\begin{equation} \label{x2eq} \lambda w^4 + xyz(x+y+z-w) = 0.\end{equation}
This family has been studied extensively by Narumiya and Shiga \cite{mmfk3sis3drp}, we will make substantial use of their results in the sequel (note, however, that our $\lambda$ is not the same as the $\lambda$ used in \cite{mmfk3sis3drp}, instead, our $\lambda$ is equal to $\mu^4$ or $\frac{u}{256}$ from \cite{mmfk3sis3drp}).

Recall from \cite[Theorem 7.1]{mslpk3s} that $\calM_{M_2}$ is the compactification of the modular curve $\Gamma_0(2)^+ \setminus \mathbb{H}$. In \cite[Section 5.4.1]{flpk3sm} it is shown that the orbifold points of orders $(2,4,\infty)$ in $\calM_{M_2}$ occur at $\lambda = (\frac{1}{256},\infty,0)$ respectively, and the K3 fibres of $\calX_2$ are smooth away from these three points. Let $U_{M_2}$ denote the open set obtained from $\calM_{M_2}$ by removing these three points. Then the restriction of $\calX_2$ to $U_{M_2}$ is an $M_2$-polarized family of K3 surfaces (in the sense of Definition \ref{defn:L-pol}).

As noted in the previous section, it follows from Theorem \ref{gfithm} that any $M_2$-polarized family of K3 surfaces $\calX_U \to U$ can be realized as the pull-back of $\calX_2$ by the generalized functional invariant map $g\colon U \to \calM_{M_2}$.

\subsection{Constructing Calabi-Yau threefolds}\label{SIconstructionsect}

In the remainder of this paper, we will use this theory to construct Calabi-Yau threefolds fibred by $M_2$-polarized K3 surfaces and study their properties. We note that, in this paper, a \emph{Calabi-Yau threefold} will always be a smooth projective threefold $\calX$ with $\omega_{\calX} \cong \calO_{\calX}$ and $H^1(\calX,\calO_{\calX})=0$. We further note that the cohomological condition in this definition implies that any fibration of a Calabi-Yau threefold by K3 surfaces must have base curve $\Proj^1$, so from this point we restrict our attention to the case where $B \cong \Proj^1$.

Recall that, by \cite[Theorem 5.10]{flpk3sm}, we already know of several Calabi-Yau threefolds with $h^{2,1} = 1$ that admit fibrations by $M_2$-polarized K3 surfaces. It is noted in \cite[Section 5.4]{flpk3sm} that the generalized functional invariant maps determining these fibrations all have a common form, given by a pair of integers $(i,j)$: the map $g$ is an $(i+j)$-fold cover ramified at two points of orders $i$ and $j$ over $\lambda = \infty$, once to order $(i+j)$ over $\lambda = 0$, and once to order $2$ over a point that depends upon the modular parameter of the threefold, where $i,j \in \{1,2,4\}$ are given in \cite[Table 1]{flpk3sm}.

The aim of this section is to extend this construction of Calabi-Yau threefolds to a more general setting. Let $g\colon \Proj^1 \to \calM_{M_2}$ be an $n$-fold cover and let $[x_1,\ldots,x_k]$, $[y_1,\ldots,y_l]$ and $[z_1,\ldots,z_m]$ be partitions of $n$ encoding the ramification profiles of $g$ over $\lambda = 0$, $\lambda = \infty$ and $\lambda = \frac{1}{256}$ respectively. Let $r$ denote the degree of ramification of $g$ away from $\lambda \in \{0,\frac{1}{256},\infty\}$, defined to be
\[r := \sum_{\mathclap{\substack{p \in \Proj^1 \\ g(p) \notin  \{0,\frac{1}{256},\infty\}}}} (e_p -1),\]
where $e_p$ denotes the ramification index of $g$ at the point $p \in \Proj^1$.

Let $\pi_2 \colon \bar{\calX}_2 \to \calM_{M_2}$ denote the threefold fibred by (singular) K3 surfaces defined by Equation \eqref{x2eq}; then $\bar{\calX}_2$ is birational to $\calX_2$. Let $\bar{\pi}_g\colon \bar{\calX}_g \to \Proj^1$ denote the normalization of the pull-back $g^*\bar{\calX}_2$. 

\begin{proposition} \label{trivialcanonicalprop} The threefold $\bar{\calX}_g$ has trivial canonical sheaf if and only if  $k+l+m-n-r=2$ and either $l = 2$ with $y_1,y_2 \in \{1,2,4\}$, or $l = 1$ with $y_1 = 8$.
\end{proposition}
\begin{proof} We begin by noting that a simple adjunction calculation shows that $\bar{\calX}_2$ has canonical sheaf $\omega_{\bar{\calX}_2} \cong \pi_2^*\calO_{\calM_{M_2}}(-1)$. We need to study the effects of the map $\bar{\calX}_g \to \bar{\calX}_2$ on this canonical sheaf.

It is an easy local computation using Equation \eqref{x2eq} to show that the pull-back $g^*\bar{\calX}_2$ is normal away from the fibres over $g^{-1}(\infty)$. To see what happens on the remaining fibres, suppose that $p \in \Proj^1$ is a point with $g(p)= \infty$ and let $y_i$ denote the order of ramification of $g$ at $p$. Then the fibre over $p$ is contained in the non-normal locus of $g^*\bar{\calX}_2$ if and only if $y_i > 1$. Away from the fibre over $p$ the normalization map $\bar{\calX_g} \to g^*\bar{\calX}_2$ is an isomorphism, whilst on the fibre over $p$ it is an $\hcf(y_i,4)$-fold cover.

With this in place, we perform two adjunction calculations. The first is for the map of base curves $g\colon \Proj^1 \to \calM_{M_2}$. As $ \calM_{M_2} \cong \Proj^1$, we find that we must have
\begin{equation}\label{basecovereq} k + l + m - n - r - 2 = 0.\end{equation}
Next, we compute $\omega_{\bar{\calX}_g}$. We find:
\[\omega_{\bar{\calX}_g} \cong \bar{\pi}_g^*\calO_{\Proj^1}\left(n+r - k - m + \sum_{i=1}^l \left(\frac{y_i}{\hcf(y_i,4)} - 1\right)\right).\]
Putting these equations together, we see that the condition that $\omega_{\bar{\calX}_g}$ is trivial is equivalent to 
\[ l - 2 + \sum_{i=1}^l \left(\frac{y_i}{\hcf(y_i,4)} - 1\right) = 0.\]
Since $l \geq 1$ and  $(\frac{y_i}{\hcf(y_i,4)} - 1)$ is nonnegative for any integer $y_i > 0$, we must therefore have either $l = 2$ and $y_i = \hcf(y_i,4)$, in which case $y_i|4$, or $l = 1$ and $y_1 = 2\,\hcf(y_1,4)$, in which case $y_1 = 8$. Together with Equation \eqref{basecovereq}, this proves the proposition.
\end{proof}

Next we will show that we can resolve most of the singularities of $\bar{\calX}_g$.

\begin{proposition} \label{KXgtrivial} If Proposition \ref{trivialcanonicalprop} holds, then there exists a projective birational morphism $\calX_g \to \bar{\calX}_g$, where ${\calX}_g$ is a normal threefold with trivial canonical sheaf and at worst $\Q$-factorial terminal singularities. Furthermore, any singularities of $\calX_g$ occur in its fibres over $g^{-1}(\frac{1}{256})$, and ${\calX}_g$ is smooth if $g$ is unramified over $\lambda = \frac{1}{256}$ \textup{(}which happens if and only if $m = n$\textup{)}. 
\end{proposition}

\begin{remark} There exist examples of maps $g\colon \Proj^1 \to \calM_{M_2}$, satisfying the conditions of this proposition and ramified over $\lambda = \frac{1}{256}$, for which the corresponding threefolds $\calX_g$ are not smooth; see Example \ref{moduliex}.
\end{remark}

\begin{proof} We prove this proposition by showing that the singularities of $\bar{\calX}_g$ may all be crepantly resolved, with the possible exception of some $\Q$-factorial terminal singularities lying in fibres over $g^{-1}(\frac{1}{256})$. 

The threefold $\bar{\calX}_2$ has six smooth curves $C_1,\ldots,C_6$ of $cA_3$ singularities which form sections of the fibration $\pi_2$. These lift to the cover  $\bar{\calX}_g$ so that, away from the fibres over $g^{-1}\{0,\frac{1}{256},\infty\}$, the threefold $\bar{\calX}_g$ also has six smooth curves of $cA_3$ singularities which form sections of the fibration. These can be crepantly resolved, so $\calX_g$ is smooth away from its fibres over $g^{-1}\{0,\frac{1}{256},\infty\}$.

Now let $\Delta_0$ be a disc in $\calM_{M_2}$ around $\lambda = 0$ and let $\Delta'_0$ be one of the connected components of its preimage under $g$. Then $g\colon \Delta'_0 \to \Delta_0$ is an $x_i$-fold cover ramified totally over $\lambda = 0$, for some $x_i$. The threefold $\bar{\calX_2}$ is smooth away from the curves $C_i$ over $\Delta_0$ and the fibre of $\pi_2\colon\bar{\calX_2} \to \calM_{M_2}$ over $\lambda = 0$ has four components meeting transversely along six curves $D_1,\ldots,D_6$, with dual graph a tetrahedron.

After pulling back to $\Delta'_0$ we see that, after resolving the pull-backs of the six curves $C_i$ of $cA_3$ singularities, the threefold $\bar{\calX}_g$ is left with six curves of $cA_{x_i-1}$ singularities in its fibre over $g^{-1}(0)$, given by the pull-backs of the curves $D_j$. Friedman \cite[Section 1]{bgd7} has shown how to crepantly resolve such a configuration using toric geometry, so $\calX_g$ is smooth over $\Delta'_0$.

Next let $\Delta_{\infty}$ be a disc in $\calM_{M_2}$ around $\lambda = \infty$ and let $\Delta'_{\infty}$ be one of the connected components of its preimage under $g$. Then $g\colon \Delta'_{\infty} \to \Delta_{\infty}$ is a $y_i$-fold cover ramified totally over $\lambda = \infty$, for some $y_i \in \{1,2,4,8\}$. 

The family $\pi_2\colon\bar{\calX}_2 \to \calM_{M_2}$ has a fibre of multiplicity four over $\lambda = \infty$ and, in addition to the six curves $C_i$ of $cA_3$ singularities forming sections of the fibration $\pi_2$, the threefold $\bar{\calX}_2$ also has four curves $E_1,\ldots,E_4$ of $cA_3$ singularities in its fibre over $\lambda = \infty$. The curves $E_j$ intersect in pairs at six points, which coincide with the six points of intersection of the curves $C_i$ with the fibre $\pi_2^{-1}(\infty)$. Thus, over $\Delta_{\infty}$ the threefold $\bar{\calX}_2$ has ten curves of $cA_3$ singularities, which meet in six triple points.

If $y_i = 1$, then $\bar{\calX}_g$ is isomorphic to $\bar{\calX}_2$ over $\Delta_{\infty}$. The ten curves of $cA_3$ singularities may be crepantly resolved by the toric method in \cite[Section 1]{bgd7}, so $\calX_g$ is smooth over $\Delta'_{\infty}$. This resolution gives three exceptional components over each curve of $cA_3$ singularities and three exceptional components over each point where three of these curves meet. The resulting singular fibre over $g^{-1}(\infty)$ has $31$ components ($3$ from each of the curves $E_j$, three from each of the six intersections between these curves, and the strict transform of the original component).

When $y_i = 2$, the four curves $E_j$ lift to four curves of $cA_1$ singularities in $\bar{\calX_g}$. The threefold $\bar{\calX_g}$ thus contains ten curves of singularities over $\Delta'_{\infty}$: six curves of $cA_3$'s coming from the pull-backs of the $C_i$ and four curves of $cA_1$'s coming from the pull-backs of the $E_j$. Away from the six points $C_i \cap E_j \cap E_k$, these ten curves of singularities are smooth and may be crepantly resolved. It therefore suffices to compute resolutions locally around these six points.

We compute such resolutions using a modification of the toric method from \cite[Section 1]{bgd7}. Locally in a neighbourhood of a point $C_i \cap E_j \cap E_k$, the singularity of $\bar{\calX_g}$ looks like $\{w^4 + t^2xy = 0\} \subset \C^4$. This corresponds to an affine toric variety with cone generated by the rays $(1,0,0)$, $(0,1,0)$, $(-1,-2,4)$. Its dual cone is generated by the rays $(0,0,1)$, $(4,0,1)$, $(0,2,1)$, and can be resolved by adding in additional rays generated by $(1,0,1)$, $(3,0,1)$, $(0,1,1)$, $(2,1,1)$, $(1,1,1)$, and $(2,0,1)$, along with faces given by iterated stellar subdivision (with rays ordered as above). This subdivision is illustrated in Figure \ref{fig:subdiv}, which shows its intersection with the slice $(a,b,1)$. As all rays added are generated by interior lattice points on a height $1$ slice, this resolution is crepant, and the iterated stellar subdivision process ensures that it will be projective. 

The corresponding configuration of exceptional divisors is shown in Figure \ref{fig:exdiv}; divisors which are part of the fibre  $g^{-1}(\infty)$ are shaded. We thus see that the resulting singular fibre over $g^{-1}(\infty)$ has $11$ components: four from the four curves $E_j$ and six from the six intersections $C_i \cap E_j \cap E_k$.

\begin{figure}
  \begin{center}
    \begin{subfigure}{0.49\linewidth}
      \begin{center}
\begin{tikzpicture}
\draw [->] (0,0)--(0,2.5);
\draw [->] (0,0)--(4.5,0);
\draw (0,2)--(4,0);
\draw (0,1)--(1,0);
\draw (1,0)--(0,2);
\draw (1,0)--(1,1);
\draw (2,0)--(0,2);
\draw (3,0)--(1,1);
\draw (3,0)--(2,1);
\draw (3,0)--(0,2);
\draw [fill] (0,0) circle [radius=0.06]; 
\draw [fill] (0,1) circle [radius=0.06]; 
\draw [fill] (0,2) circle [radius=0.06]; 
\draw [fill] (1,0) circle [radius=0.06]; 
\draw [fill] (1,1) circle [radius=0.06]; 
\draw [fill] (2,0) circle [radius=0.06]; 
\draw [fill] (2,1) circle [radius=0.06];  
\draw [fill] (3,0) circle [radius=0.06];    
\draw [fill] (4,0) circle [radius=0.06]; 
\node [right] at (0,2.2) {$(0,2,1)$};
\node [below] at (4,0) {$(4,0,1)$};
\node [below] at (0,0) {$(0,0,1)$};
\end{tikzpicture}
      \end{center}
      \caption{Subdivision of the slice $(a,b,1)$}
      \label{fig:subdiv}
    \end{subfigure}
    \begin{subfigure}{0.49\linewidth}
      \begin{center}
\begin{tikzpicture}
\draw [lightgray, fill=lightgray] (0,1.5)--(0.5,1)--(1,1.5)--(0,2.5);
\draw [lightgray, fill=lightgray] (1,1.5)--(3.5,1.5)--(4.5,2.5)--(0,2.5);
\draw [lightgray, fill=lightgray] (4.5,1.5)--(4,1)--(3.5,1.5)--(4.5,2.5);
\draw [lightgray, fill=lightgray] (1.5,1.5)--(2,1)--(2.5,1)--(3,1.5);
\draw (0.5,0)--(0.5,1);
\draw (0.5,1)--(0,1.5);
\draw (0.5,1)--(1,1.5);
\draw (1,1.5)--(0,2.5);
\draw (1,1.5)--(3.5,1.5);
\draw (1.5,1.5)--(2,1);
\draw (2,1)--(2,0);
\draw (2,1)--(2.5,1);
\draw (2.5,1)--(2.5,0);
\draw (2.5,1)--(3,1.5);
\draw (3.5,1.5)--(4,1);
\draw (3.5,1.5)--(4.5,2.5);
\draw (4,1)--(4,0);
\draw (4,1)--(4.5,1.5);

\end{tikzpicture}
      \end{center}
      \caption{Corresponding exceptional divisors}
      \label{fig:exdiv}
    \end{subfigure}
\caption{Toric resolution in a neighbourhood of a point $C_i \cap E_j \cap E_k$ when $y_i = 2$}
  \end{center}
\end{figure}
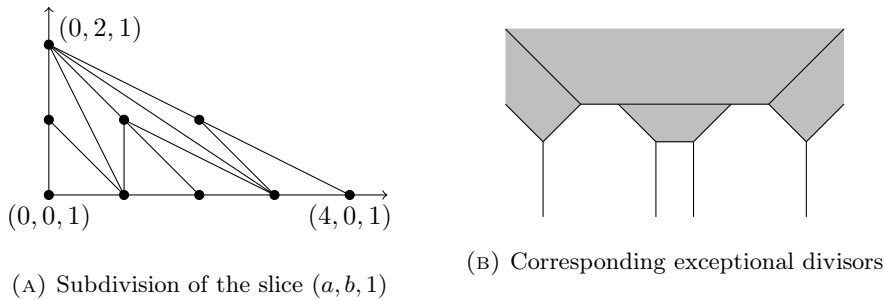

Finally, if $4|y_i$, the lifts of the four curves $E_j$ are smooth in $\calX_g$. Over $\Delta'_{\infty}$, the threefold $\bar{\calX_g}$ thus contains only the six curves of $cA_3$ singularities coming from the pull-backs of the curves $C_i$. These may be crepantly resolved without adding any new components to the fibre of $\bar{\calX_g}$ over $g^{-1}(\infty)$. Thus we see that, in all cases, the threefold $\calX_g$ is smooth over $\Delta'_{\infty}$.

Finally, let $\Delta_{\frac{1}{256}}$ be a disc in $\calM_{M_2}$ around $\lambda = \frac{1}{256}$ and let $\Delta'_{\frac{1}{256}}$ be one of the connected components of its preimage under $g$. Then $g\colon \Delta'_{\frac{1}{256}} \to \Delta_{\frac{1}{256}}$ is an $z_i$-fold cover ramified totally over $\lambda = \frac{1}{256}$, for some $z_i$. 

The threefold $\bar{\calX}_2$ is smooth over $\Delta_{\frac{1}{256}}$ away from the six curves $C_i$ of $cA_3$'s forming sections of the fibration, but its fibre over $\lambda = \frac{1}{256}$ has an additional isolated $A_1$ singularity. Upon proceeding to the $z_i$-fold cover $\Delta'_{\frac{1}{256}} \to \Delta_{\frac{1}{256}}$, this becomes an isolated terminal singularity of type $cA_{z_i-1}$ in $\bar{\calX}_g$. 

Thus $\calX_g$ is smooth away from its fibres over $g^{-1}(\frac{1}{256})$, where it can have isolated terminal singularities. By \cite[Theorem 6.25]{bgav}, we may further assume that $\calX_g$ is $\Q$-factorial. To complete the proof, we note that if $g$ is a local isomorphism over $\Delta_{\frac{1}{256}}$, then $\calX_g$ is also smooth over $g^{-1}(\frac{1}{256})$ and thus smooth everywhere.
\end{proof}

Let $\pi_g\colon \calX_g \to \Proj^1$ denote the fibration induced on $\calX_g$ by the map $\bar{\pi}_g \colon \bar{\calX}_g \to \Proj^1$. By construction, the restriction of $\pi_g$ to the smooth fibres is an $M_2$-polarized family of K3 surfaces, in the sense of \cite[Definition 2.1]{flpk3sm}.  Moreover, we have:

\begin{proposition}\label{Xgsimplyconnected} Let $\calX_g$ be a threefold as in Proposition \ref{KXgtrivial}. Then the cohomology $H^1(\calX_g,\calO_{\calX_g}) = 0$.
\end{proposition}
\begin{proof} Since $H^1(\Proj^1,\calO_{\Proj^1}) = 0$, the proposition will follow immediately from the Leray spectral sequence if we can prove that $R^1(\pi_g)_*\calO_{\calX_g} = 0$.

To show this, we note that $\calX_g$ is a normal projective threefold with at worst terminal singularities and the canonical sheaf $\omega_{\calX_g} \cong \calO_{\calX_g}$ is locally free. So we may apply the torsion-freeness theorem of Koll\'{a}r \cite[Theorem 10.19]{smaf} to see that $R^1(\pi_g)_*\calO_{\calX_g}$ is a torsion-free sheaf on $\Proj^1$. Moreover, since $H^1(X,\calO_X) = 0$ for a generic fibre $X$ of $\pi_g\colon \calX_g \to \Proj^1$, the sheaf $R^1(\pi_g)_*\calO_{\calX_g}$ also has trivial generic fibre. We must therefore have $R^1(\pi_g)_*\calO_{\calX_g} = 0$.
\end{proof}

\begin{corollary} \label{XgCY3}  Let $\calX_g$ be a threefold as in Proposition \ref{KXgtrivial}. If $\calX_g$ is smooth, then $\calX_g$ is a Calabi-Yau threefold.

Conversely, let $\calX \to \Proj^1$ be a Calabi-Yau threefold fibred by K3 surfaces, let $U \subset \Proj^1$ denote the open set over which the fibres of $\calX$ are smooth and let $\calX_U$ denote the restriction of $\calX$ to $U$. Suppose that $\calX_U$ is an $M_2$-polarized family of K3 surfaces \textup{(}in the sense of Definition \ref{defn:L-pol}\textup{)} and that the N\'{e}ron-Severi group of a general fibre of $\calX_U$ is isometric to $M_2$. Then $\calX$ is birational to a threefold $\calX_g$ as in Proposition \ref{KXgtrivial}.
\end{corollary}
\begin{proof}  To prove the first statement note that, by Proposition \ref{KXgtrivial}, $\calX_g$ has trivial canonical bundle. Moreover, $H^1(\calX_g,\calO_{\calX_g}) = 0$ by Proposition \ref{Xgsimplyconnected}. So $\calX_g$ is Calabi-Yau.

The converse statement follows from the fact that, if $g \colon \Proj^1 \to \calM_{M_2}$ denotes the generalized functional invariant of $\calX$, then Theorem \ref{gfithm} shows that $\calX$ and $\calX_g$ are isomorphic over the open set $U$.
\end{proof}

\begin{example} \label{ijex} We now explain how the five Calabi-Yau threefolds fibred by $M_2$-polarized K3 surfaces from \cite[Theorem 5.10]{flpk3sm} fit into this picture. As noted in \cite[Section 5.4]{flpk3sm}, in each of these cases the generalized functional invariant $g\colon \Proj^1 \to \calM_{M_2}$ has the special form 
\[g\colon (s:t) \mapsto \lambda = \frac{As^{i+j}}{t^i(s-t)^j},\]
where $(s:t)$ are coordinates on $\Proj^1$, $A \in \C$ is a modular parameter for the threefold, and  $i,j \in \{1,2,4\}$ are as listed in \cite[Table 1]{flpk3sm}.
 
In our notation from above, this generalized functional invariant map $g$ has $(k,l,m,n,r) = (1,2,i+j,i+j,1)$, $[x_1] = [i+j]$, $[y_1,y_2] = [i,j]$, and $[z_1,\ldots,z_{i+j}] = [1,\ldots,1]$. It follows immediately from Proposition \ref{KXgtrivial} and Corollary \ref{XgCY3} that $\calX_g$ is a smooth Calabi-Yau threefold  for each choice of $i$, $j$, as we should expect. The reason for the special form of the generalized functional invariants $g$ appearing in these cases will be discussed later in Remark \ref{ijrem}.

In particular, we see that the $M_2$-polarized K3 fibration on the quintic mirror threefold appears as a special case of this construction, with $(i,j) = (1,4)$. Its generalized functional invariant $g$ therefore has $(k,l,m,n,r) = (1,2,5,5,1)$, $[x_1] = [5]$, $[y_1,y_2] = [1,4]$, and $[z_1,\ldots,z_5] = [1,1,1,1,1]$.
\end{example}

\section{Hodge Numbers}\label{Hodgesect}

This enables us to construct a large class of Calabi-Yau threefolds $\calX_g$ admitting fibrations by $M_2$-polarized K3 surfaces. We next study the Hodge numbers of these Calabi-Yau threefolds. 

\begin{remark}\label{l=2rem} It is easy to see here that the case $l = 1$, $y_1 = 8$ is a smooth degeneration of the case $l = 2$, $(y_1,y_2) = (4,4)$, corresponding to restriction to a sublocus in moduli. Therefore, when discussing the Hodge numbers of the  Calabi-Yau threefolds $\calX_g$, to avoid pathological cases we will restrict to the case $l = 2$. In this case we have $y_1,y_2 \in \{1,2,4\}$ and $k + m - n - r = 0$.
\end{remark}

\subsection{Leray filtrations and local systems}
\label{sec:Betti}

We begin by developing some general theoretical results that apply to any K3-fibred threefolds, then specialize these to the case that interests us. So let $\pi\colon \calX \to B$ be a smooth projective threefold fibred by K3 surfaces over a smooth complete base curve $B$. Our aim is to compute the Betti numbers $b_i(\calX)$ of $\calX$. 

Let us set up some notation. Denote the fibre of $\calX$ over $p \in B$ by $X_p$, let $\Sigma \subset B$ denote the set of points $p$ over which the fibre $X_p$ is not smooth, and let $U = \Proj^1 - \Sigma$. Let $j\colon U \hookrightarrow B$ denote the natural embedding and let $\pi_U$ denote the restriction of $\pi$ to $\pi^{-1}(U)$. 

With this in place, our first lemma computes $b_2(\calX)$ and $b_4(\calX)$. It should be thought of as a version of the Shioda-Tate formula for K3 surface fibrations. 

\begin{lemma}\label{lemma:leray1,1}
With notation as above, let $\rho_p$ denote the number of irreducible components of $X_p$ and let $\rho_h$ be the rank of the subgroup of $H^2(X_p,\mathbb{C})$, for $p \in U$, that is fixed by the action of monodromy. Then we have
\[b_2(\calX) = b_{4}(\calX) = 1 + \rho_h + \sum_{p\in B} (\rho_p - 1).\]
\end{lemma}
\begin{proof}
We begin by assuming that each singular fibre of $\calX$ is normal crossings. We address the non-normal crossings situation at the end of the proof.

We compute $b_4(\calX)$ using the Leray spectral sequence. By \cite[Corollary 15.15]{htdcl2cpm}, this spectral sequence degenerates at the $E_2$ term, so we may write
\[b_4(\calX) = h^0(B, R^4\pi_*\mathbb{C}) +  h^1(B, R^3\pi_*\mathbb{C}) +  h^2(B , R^2\pi_*\mathbb{C}).\]
We will deal with each term in succession.

To compute these terms we will make repeated use of the following fact. By \cite[Proposition 15.12]{htdcl2cpm}, for each $n \geq 0$ there is a surjective map of constructible sheaves 
\[R^n \pi_*\mathbb{C} \longrightarrow j_* R^n (\pi_U)_*\mathbb{C}.\]
The kernel of this map is a skyscraper sheaf $\calS_n$ supported on $\Sigma$, so we have a short exact sequence
\begin{equation}\label{eq:skyses}
0 \longrightarrow \calS_n \longrightarrow R^n \pi_*\mathbb{C} \longrightarrow j_* R^n(\pi_U)_*\mathbb{C} \longrightarrow 0
\end{equation}

Now we begin by computing $H^0(B, R^4\pi_*\mathbb{C})$. Since $\pi_U$ is a topologically trivial fibration, we have that $R^4 (\pi_U)_* \mathbb{C}$ is a constant sheaf with stalk $\C$, so $j_* R^4(\pi_U)_*\mathbb{C}$ is also. Taking cohomology in \eqref{eq:skyses} (for $n = 4$), we thus obtain a short exact sequence of cohomology groups
\[0 \longrightarrow H^0(B,\calS_4) \longrightarrow H^0(B, R^4\pi_*\mathbb{C}) \longrightarrow \mathbb{C} \longrightarrow 0.\]
Since $\calS_4$ is a skyscraper sheaf supported on $\Sigma$, it suffices to compute $H^0(B,\calS_4)$ locally around each point of $\Sigma$. So choose a small neighbourhood $U_s$ around each $s \in \Sigma$. The Clemens contraction theorem implies that $H^0(U_s,R^4\pi_*\mathbb{C}) = H^{4}(\pi^{-1}(U_s),\mathbb{C}) \cong H^{4}(X_s,\mathbb{C})$. 

Since $X_s$ is normal crossings,  the rank of $H^{4}(X_s,\mathbb{C})$ can be computed using the Mayer-Vietoris spectral sequence, as described by Griffiths and Schmid in \cite[Section 4]{rdhtdtr}, as follows. Let $\{S_i\}$ denote the set of irreducible components of $X_s$, then define $(X_s)_{i_1,\dots,i_p} = \bigcap_{i_0 \dots, i_p} S_{i_j}$ for a disjoint set of indices $i_0,\dots,i_p$ and let $X_s^{[p]} =\coprod_{i_0< \dots < i_p} (X_s)_{i_0,\dots,i_p}$. The $E_1$ term of the Mayer-Vietoris spectral sequence is given by $E_1^{p,q} = H^q(X_s^{[p]},\mathbb{C})$.

This spectral sequence degenerates at $E_2$ and converges to $H^{p+q}(X_s,\mathbb{C})$. Computing $H^{4}(X_s,\mathbb{C})$ thus reduces to computing $E_1^{i,4-i} = H^{4-i}(X^{[i]}_s,\mathbb{C})$ for each $0 \leq i \leq 4$. Note that this term vanishes for $i \neq 0$ for dimension reasons, so it follows that $E_1^{i,4-i} = 0$ if $i \neq 0$ and $E_1^{0,4} = H^{4}(X^{[0]}_s,\mathbb{C}) \cong \mathbb{C}^{\rho_s}$. We therefore see that $H^{4}(X_s,\mathbb{C}) \cong \mathbb{C}^{\rho_s}$, which gives $H^0(U_s,\calS_4) \cong \C^{\rho_s -1}$. 

Putting everything together, we obtain an exact sequence
\[0 \longrightarrow \bigoplus_{s \in \Sigma} \C^{\rho_s -1} \longrightarrow H^0(B, R^4\pi_*\mathbb{C}) \longrightarrow \mathbb{C} \longrightarrow 0,\]
and taking dimensions gives
\[h^0(B,R^4f_* \mathbb{C}) = 1+ \sum_{s\in \Sigma} (\rho_s-1).\]

Next we compute $H^1(B, R^3\pi_*\mathbb{C})$. Taking cohomology of the exact sequence \eqref{eq:skyses} (for $n = 3$), we obtain an isomorphism
\[H^1(B, R^3 \pi_*\mathbb{C}) \cong H^1(B, j_*R^3 (\pi_U)_*\mathbb{C}).\]
But $R^3 (\pi_U)_*\mathbb{C}$ is trivial, since every fibre of $\pi_U$ is a smooth K3 surface, therefore $j_* R^3 (\pi_U)_*\mathbb{C}$ is the trivial sheaf on $B$ and $H^1(B, j_* R^3(\pi_U)_*\mathbb{C}) = 0$.

To compute $H^2(B, R^2\pi_*\mathbb{C})$, we use the following standard exact sequence
\[0 \rightarrow j_! R^2(\pi_U)_*\mathbb{C} \rightarrow R^2\pi_*\mathbb{C} \rightarrow i_* ((R^2\pi_*\mathbb{C})|_{\Sigma}) \rightarrow 0\]
where $i\colon \Sigma \hookrightarrow B$ denotes the inclusion. Again, $i_* ((R^2\pi_*\mathbb{C})|_{\Sigma})$ is a skyscraper sheaf supported on $\Sigma$, so taking cohomology we deduce that 
\[H^2(B,j_! R^2(\pi_U)_*\mathbb{C} ) \cong H^2(B, R^2\pi_*\mathbb{C}).\]
It is then known that $H^2(B, j_! R^2(\pi_U)_* \mathbb{C})$ is just $H^2_c(U, R^2 (\pi_U)_*\mathbb{C})$, where $H^2_c$ denotes compactly supported cohomology. Verdier duality can then be applied to give that
\[H^2_c(U, R^2(\pi_U)_*\mathbb{C}) \cong H^0(U, \mathscr{H}\!\mathit{om}(R^2 (\pi_U)_* \mathbb{C},\mathbb{C})).\]
Finally, we recognize that classes in $H^0(U, \mathscr{H}\!\mathit{om}(R^2(\pi_U)_*\mathbb{C},\mathbb{C}))$ are in bijection with classes in $R^2 (\pi_U)_* \mathbb{C}$ which are fixed by the monodromy representation. Therefore
\[h^2(B, R^2\pi_*\mathbb{C}) = h^2(X_p,\mathbb{C})^{\pi_1(U, p)} = \rho_h\]
as required. This completes the proof when all singular fibres of $\calX$ have normal crossings. 

If $\calX$ does not have normal crossings singular fibres, Hironaka's theorem on embedded resolutions says that we may repeatedly blow up smooth loci in $\calX$ until its fibres are normal crossings. Let $\widetilde{\calX} \rightarrow \calX$ be such a blow-up, and assume that we have blown up $r$ times. Then $h^{4}(\widetilde{\calX}) = h^4(\calX) + r$. Furthermore, each blow-up adds only one component to a singular fibre, so if $\tilde{\pi}\colon \widetilde{\calX} \to B$ is the induced fibration  and $\sigma_p$ is the number of irreducible components in $\tilde{\pi}^{-1}(p)$, then 
\[\sum_{p \in B} ( \rho_p - 1) + r = \sum_{p \in B}(\sigma_p - 1).\]
Therefore the result for $\calX$ follows from the result for $\widetilde{\calX}$.
\end{proof}

\begin{remark}
This proof can be repeated for any fibration whose fibres are generically smooth and satisfy $H^1(X_p,\mathbb{C}) = 0$.
\end{remark}

Next we use similar methods to prove a result that will allow us to compute $b_3(\calX)$.

\begin{lemma}
\label{lemma:fibers}
With notation as above, if every fibre $X_p$ of $\pi$ is either:
\begin{enumerate}
\item a K3 surface with at worst ADE singularities, or
\item a normal crossings union of smooth surfaces $S_i$, with $H^3(S_i,\mathbb{C}) = 0$,
\end{enumerate}
then $b_3(\calX) = h^1(B, j_* R^2(\pi_U)_* \mathbb{C})$.
\end{lemma}

\begin{proof}
As in the proof of Lemma \ref{lemma:leray1,1}, we use the degeneration of the Leray spectral sequence for the fibration $\pi\colon \calX \rightarrow B$. This tells us that
\[b_3(\calX) = h^0(B, R^3\pi_*\mathbb{C}) +h^1(B, R^2\pi_*\mathbb{C}) + h^2(B, R^1\pi_*\mathbb{C}).\]
As usual, we have the short exact sequence \eqref{eq:skyses}, which reads
\[0 \longrightarrow \mathcal{S}_n \rightarrow R^n \pi_*\mathbb{C} \longrightarrow j_*R^n (\pi_U)_* \mathbb{C} \longrightarrow 0.\]
Recall that $\mathcal{S}_n$ is a skyscraper sheaf supported on $\Sigma$.

Since $R^1 (\pi_U)_*\mathbb{C}$ is trivial, so is $j_* R^1 (\pi_U)_* \mathbb{C}$. Therefore, taking cohomology of the above sequence (for $n=1$), we obtain that $H^2(B, R^1\pi_*\mathbb{C})$ must also be trivial. A similar argument can be used to show that $H^1(B, R^2\pi_*\mathbb{C})$ is isomorphic to $H^1(B, j_* R^2(\pi_U)_*\mathbb{C})$.

It remains to show that $H^0(B, R^3\pi_*\mathbb{C})$ is trivial. Noting that $j_*R^3 (\pi_U)_* \mathbb{C}$ is trivial, taking cohomology in the exact sequence above (for $n=3$) gives an isomorphism $H^0(B, R^3\pi_*\mathbb{C}) \cong H^0(B,\calS_3)$ (this last group is called $A$ in \cite{htdcl2cpm}).

Since $\calS_3$ is a skyscraper sheaf on $\Sigma$, it suffices to show that $H^0(B,\calS_3)$ is trivial locally around each point of $\Sigma$. So let $s \in \Sigma$ and let $U_s \subset B$ be a small neighbourhood of $s$. The Clemens contraction theorem implies that $H^0(U_s,R^3\pi_*\C) = H^3(\pi^{-1}(U_s),\C) \cong H^3(X_s,\C)$, which is trivial if $X_s$ is a fibre satisfying case (1) of the lemma, so $ H^0(U_s,\calS_3)$ is also trivial for such fibres.

We may therefore restrict ourselves fibres satisfying case (2) of the lemma. Zucker \cite[Section 15]{htdcl2cpm} shows how to compute $H^0(U_s,\calS_3)$ in a neighbourhood of such fibres, as follows. Let $s \in \Sigma$ be a point over which the fibre satisfies the conditions of case (2) and let $\Delta_s \subset U_s$ be a small closed disc around $s$. Let  ${\calX}_{\Delta_s} = {\pi}^{-1}({\Delta_s})$ and let $\partial {\calX}_{\Delta_s}$ be the boundary of ${\calX}_{\Delta_s}$. Then Zucker \cite[Section 15]{htdcl2cpm} proves that $H^0(U_s,\calS_3)$ may be computed as the image of a morphism of mixed Hodge structures,
\[H^0(U_s,\calS_3) = \mathrm{im}\left( \phi_s\colon H^3({\calX}_{\Delta_s}, \partial {\calX}_{\Delta_s}) \rightarrow H^3({X}_s)\right),\]
from which it follows that $H^0(U_s,\calS_3)$ admits a pure Hodge structure of weight $3$. The exact definition of the map $\phi_s$ appearing here will not concern us, as we only need to use the fact that it is a morphism of mixed Hodge structures; the interested reader may refer to \cite[Section 15]{htdcl2cpm} for more details. 

To show that this image is trivial, we use the Mayer-Vietoris spectral sequence, as described in the proof of Lemma \ref{lemma:leray1,1}. Recall that this is a spectral sequence which degenerates to $H^{p+q}(X_s,\mathbb{C})$ at the $E_2$ term and satisfies $E_1^{p,q} = H^q(X_s^{[p]},\mathbb{C}),$
where $X_s^{[p]}$ is the disjoint union of all codimension $p$ strata in $X_s$.  Its graded pieces $\Gr^W_i$ are the weight-graded pieces of the functorial mixed Hodge structure on $X_s$. Thus only $H^3(X_s^{[0]},\mathbb{C})$, $H^2(X_s^{[1]},\mathbb{C})$ and $H^1(X_s^{[2]},\mathbb{C})$ may contribute to $H^3(X_s,\mathbb{C})$. Moreover, by the condition that $H^3(S_i,\mathbb{Q}) = 0$ we obtain $\Gr^W_3 H^3(X_s,\mathbb{C})= H^3(X_s^{[0]},\mathbb{C}) = 0$. 

Zucker \cite[Section 15]{htdcl2cpm} notes that the weight filtration on $H^3({\calX}_{\Delta_s}, \partial {\calX}_{\Delta_s})$ has $W_i = 0$ for $i \leq 2$. By strictness, we thus see that 
\[\mathrm{im}(\phi_s) \cap W_2(H^3({X}_s,\mathbb{C})) = 0\]
and
\[\phi_s(W_i H^3({\calX}_{\Delta_s}, \partial {\calX}_{\Delta_s})) = \mathrm{im}(\phi_s) \cap W_i(H^3(X_s)) = \mathrm{im}(\phi_s) \cap W_3(H^3(X_s))\]
for all $i\geq 3$. So in particular
\[\mathrm{im}(\phi_s) = \phi_s(W_3 H^3(\calX_{\Delta_s}, \partial \calX_{\Delta_s})) \subset  W_3(H^3(X_s))\]
and the map
\[W_3(H^3(X_s)) \longrightarrow \Gr^W_3 (H^3(X_s))\]
is injective on the image of $\phi_s$. Thus the image of the induced map 
\[W_3H^3(\calX_{\Delta_s}, \partial \calX_{\Delta_s})\stackrel{\phi_s}{\longrightarrow} W_3(H^3(X_s)) \longrightarrow \Gr^W_3 (H^3(X_s)) =0\]
is equal to the image of $\phi_s$, so $H^0(U_s,\calS_3)$ is trivial. 
\end{proof}

\subsection{The Hodge number $h^{1,1}$}

Using Lemma \ref{lemma:leray1,1}, the Hodge number $h^{1,1}$ is relatively easy to compute. 

\begin{proposition} \label{Xgh11} Let $\calX_g$ be a Calabi-Yau threefold as in Corollary \ref{XgCY3} and suppose that $g^{-1}(\infty)$ consists of two points \textup{(}so that $l=2$\textup{)}. Then 
\[h^{1,1}(\calX_g) = 20 + \sum_{i=1}^{k}(2x_i^2 + 1) + c_1 + c_2, \]
where $[x_1,\ldots,x_k]$ is the partition of $n$ encoding the ramification profile of $g$ over $\lambda = 0$ and $c_1$, $c_2$ are given in terms of the partition $[y_1,y_2]$ of $n$ encoding the ramification profile of $g$ over $\lambda = \infty$ by $c_j = 30$ \textup{(}resp. $10$, $0$\textup{)} if and only if $y_j = 1$ \textup{(}resp. $2$, $4$\textup{)}.
\end{proposition}
\begin{proof} Using Lemma \ref{lemma:leray1,1} we see that 
\[h^{1,1}(\calX_g) = b_2(\calX_g) = \rho_h + 1 + \sum_{p\in \Proj^1} (\rho_p - 1).\] 

We first compute $\rho_h$. Since $\calX_g$ is a compactification of the pullback of $\calX_2$ by $g$, it follows that $\rho_h \geq 19$. Suppose for a contradiction that $\rho_h = 20$. Note that the monodromy group of $\calX_g$ must be infinite, as it is a finite index subgroup of the monodromy group of $\calX_2$, which is infinite by \cite[Theorem 4.2]{mmfk3sis3drp}. So we may apply \cite[Theorem 3.1]{hgk3s} to see that every smooth fibre of $\calX_g$ must have Picard rank $20$. But this is clearly not the case. Therefore $\rho_h = 19$.

It remains to compute the number of irreducible components $\rho_p$ of the fibre of $\calX_g$ over each choice of $p \in \Proj^1$. Such fibres are irreducible over all points $p \in \Proj^1$ with $g(p) = \lambda \notin \{0,\infty\}$, so we only need to consider the points $p$ outside this set. First suppose that $p$ is a point with $g(p) = 0$ and let $x$ denote the order of ramification of $g$ at $p$. The fibre of $\calX_2$ over $\lambda = 0$ is semistable, with four components arranged as a tetrahedron. So the fibre of $\calX_g$ over $p$ is isomorphic to the resolution of the pull-back of such a fibre by a base change $t \mapsto t^x$, where $t$ is a local coordinate around $p \in \Proj^1$.

Pull-backs of semistable fibres of this kind were computed by Friedman \cite[Section 1]{bgd7}. By \cite[Proposition 1.2]{bgd7}, we see that the fibre of $\calX_g$ over $p$ has
\begin{itemize}
\item $4$ components that are strict transforms of the original $4$,
\item $6(x-1)$ new components arising from the blow-ups of the six curves of $cA_{x-1}$ singularities that occur along the pull-backs of the six edges of the tetrahedron, and
\item $2(x-1)(x-2)$ new components arising from the blow-ups of the pull-backs of the four corners of the tetrahedron.
\end{itemize}
Summing, we obtain $(2x^2 + 2)$ components in the fibre over $p$.

Finally, we consider a fibre of $\calX_g$ over a point $p$ with $g(p) = \infty$. Let $y$ denote the order of ramification of $g$ at $p$; by Proposition \ref{trivialcanonicalprop} we have $y \in \{1,2,4\}$. In each case, the fibre of $\calX_g$ over $p$ was computed explicitly in the proof of Proposition \ref{KXgtrivial}. In particular, we found that it has $31$ (resp. $11$, $1$) components when $y = 1$ (resp. $2$, $4$).

Summing over all fibres in $\Sigma$, we find that 
\[\sum_{p\in\Sigma} (\rho_p - 1) = \sum_{i=1}^{k}(2x_i^2 + 1) + c_1 + c_2,\]
where $x_i$ and $c_j$ are as in the statement of the proposition. Adding in $\rho_h + 1 = 20$ gives the desired result.
\end{proof}

\subsection{The Hodge number $h^{2,1}$}\label{h21sect}

Next we use Lemma \ref{lemma:fibers} to compute the Hodge number $h^{2,1}$. We begin by showing that this computation can be reduced to the calculation of the monodromy around the singular fibres.

Let $\mathbb{V}$ be an irreducible $\mathbb{Q}$-local system on a quasi-projective curve $U$ and let $j\colon U \hookrightarrow B$ be the canonical injection of $U$ into its smooth closure. Associated to $\mathbb{V}$ and a base point $p \in U$, we have a representation $\rho\colon \pi_1(U,p) \rightarrow \GL(\mathbb{V}_p)$, where $\mathbb{V}_p$ is the fibre of $\mathbb{V}$ at $p$. 

Denote the points in $B - U$ by $\{q_1,\ldots, q_s\}$. Via this representation, to each $q_i$ we may associate a local monodromy matrix $\gamma_i$, coming from a counterclockwise loop about $q_i$. This allows us to associate an integer
\[R(q_i) := \rank \mathbb{V}_p - \dim(\mathbb{V}_p^{\gamma_i})\]
to each $q_i$, where $\mathbb{V}_p^{\gamma_i}$ is the subspace of elements of $\mathbb{V}_p$ that are fixed under the action of $\gamma_i$

With this in place, we may compute $h^1(B,j_*\mathbb{V})$ using the following variation on Poincar\'{e}'s formula in classical topology, due to del Angel, M\"{u}ller-Stach, van Straten and Zuo \textup{\cite[Proposition 3.6]{hca1pfcy3f}}:
\begin{equation}\label{eq:katz} h^1(B,j_*\mathbb{V}) = \sum_{i=1}^n R(q_i) + 2(g(B) - 1) \rank(\mathbb{V}).\end{equation}

As a result of this formula and Lemma \ref{lemma:fibers}, if the singular fibres of $\pi$ satisfy the assumptions of Lemma \ref{lemma:fibers} and we know the local monodromy matrices, then we can easily deduce the Betti number $b_3(\calX)$ of a K3 surface fibred threefold $\calX$. These conditions are satisfied by the examples discussed in Section \ref{constructionsect}.

\begin{example}
\label{ex:X2}
Let $\mathcal{X}_2 \to \calM_{M_2}$ be the K3-fibred threefold discussed in Section \ref{SIsect}. Recall that $\mathcal{X}_2$ has three singular fibres, over the points $(q_1,q_2,q_3) = (0, \frac{1}{256}, \infty)$, and that the family of K3 surfaces over $U_{M_2} := \calM_{M_2} - \{q_1,q_2,q_3\}$ is an $M_2$-polarized family of K3 surfaces (in the sense of \cite[Definition 2.1]{flpk3sm}). 

If $\pi_U$ denotes the restriction of the fibration $\mathcal{X}_2 \to \calM_{M_2}$ to $U_{M_2}$, then $R^2(\pi_U)_*\mathbb{Q}$ is a $\Q$-local system on $U_{M_2}$. It is easy to see from the explicit description in the proof of Proposition \ref{KXgtrivial} that the singular fibres of $\mathcal{X}_2$ are either nodal K3 surfaces or normal crossings unions of smooth rational surfaces. Thus we may apply Lemma \ref{lemma:fibers} to deduce that $b_3(\mathcal{X}_2) = h^1( \calM_{M_2},\,j_*R^2(\pi_U)_*\mathbb{C})$, where $j\colon U_{M_2} \to \calM_{M_2}$ denotes the inclusion.

It remains to compute this cohomology. The discussion of \cite[Section 2.1]{flpk3sm} gives a splitting of $R^2(\pi_U)_*\mathbb{Q}$ as a direct sum of two irreducible $\Q$-local systems 
\[R^2(\pi_U)_*\mathbb{Q} = \mathcal{NS}(\mathcal{X}_2) \oplus \mathcal{T}(\mathcal{X}_2),\]
where $\mathcal{NS}(\mathcal{X}_2)$ consists of those classes which are in $\NS(X_p) \otimes \mathbb{Q}$ for every smooth fibre $X_p$ of $\mathcal{X}_2$, and $\mathcal{T}(\mathcal{X}_2)$ is the orthogonal complement of $\mathcal{NS}(\mathcal{X}_2)$. In our situation, $\mathcal{NS}(\mathcal{X}_2)$ is a trivial rank $19$ local system, and $\mathcal{T}(\mathcal{X}_2)$ is an irreducible local system of rank $3$. 

 We therefore have
\[H^1( \calM_{M_2},j_*R^2(\pi_U)_*\mathbb{Q}) = H^1( \calM_{M_2},j_*\mathcal{T}(\mathcal{X}_2))  \oplus H^1( \calM_{M_2},j_*\mathcal{NS}(\mathcal{X}_2)),\]
and triviality of $\mathcal{NS}(\mathcal{X}_2)$ gives
\[b_3(\calX_2) = h^1( \calM_{M_2},j_*R^2(\pi_U)_*\mathbb{C}) = h^1( \calM_{M_2},j_*\mathcal{T}(\mathcal{X}_2)).\]

We will compute this last cohomology group using Equation \eqref{eq:katz}. According to \cite[Section 5]{mmfk3sis3drp}, the Picard-Fuchs equation of the family of K3 surfaces $\mathcal{X}_2$ is hypergeometric of type $_3F_2(\frac{1}{4},\frac{2}{4},\frac{3}{4};1,1;256\lambda)$. From this, we may use a theorem of Levelt (\cite[Theorem 1.1]{leveltthesis}, see also \cite[Theorem 3.5]{mhff}) to compute that the global monodromy representation of $\mathcal{T}(\mathcal{X}_2) \otimes \mathbb{R}$ is given by the monodromy matrices
\[A = \left(\begin{array}{ccc}
0 & 0 & -1 \\
1 & 0 & -1 \\
0 & 1 & -1
\end{array}\right), \quad
B^{-1} = \left(\begin{array}{ccc}
3 & 1 & 0 \\
-3 & 0 & 1 \\
1 & 0 & 0
\end{array}\right),\quad
A^{-1}B = \left(\begin{array}{ccc}
1 & 0 & -4 \\
0 & 1 & 2 \\
0 & 0 & -1
\end{array}\right)\]
around $\lambda = \infty$, $\lambda = 0$ and $\lambda = \frac{1}{256}$ respectively (here the names of the matrices have been chosen to agree with \cite{mhff}).  Thus the local system $\mathcal{T}(\mathcal{X}_2) \otimes \mathbb{C}$ has local monodromy matrices given by 
\[\gamma_\infty = \left(\begin{matrix} \sqrt{-1} & 0 & 0 \\ 0 & -\sqrt{-1} & 0 \\ 0 & 0 & -1 \end{matrix} \right),\quad \gamma_0 = \left(\begin{matrix} 1 & 1 & 0 \\ 0 & 1 & 1 \\ 0 & 0 & 1 \end{matrix} \right),\quad \gamma_{\frac{1}{256}} = \left(\begin{matrix} -1 & 0 & 0 \\ 0 & 1 & 0 \\ 0 & 0 & 1 \end{matrix} \right)\]

Using Equation \eqref{eq:katz}, we can therefore compute that
\[b_3(\mathcal{X}_2)  = h^1( \calM_{M_2},j_*\mathcal{T}(\mathcal{X}_2))= 3 + 1 + 2 - 2\cdot 3 = 0\]
Local systems $\mathbb{V}$ satisfying $h^1(\mathbb{P}^1, j_*\mathbb{V})=0$ are called \emph{extremal local systems} in \cite{msfm}.
\end{example}

\begin{remark} Note that we choose not to use the explicit monodromy matrices computed in \cite[Section 4]{mmfk3sis3drp} for this calculation. This is because the method used to compute monodromy matrices in \cite{mmfk3sis3drp} contains a sign ambiguity, corresponding to the choice of primitive fourth root of unity in the transformation \cite[(4.1)]{mmfk3sis3drp}. Making the opposite choice has the effect of applying an antisymplectic involution on the fibres, which multiplies the monodromy matrices $\gamma_{\infty}$ and $\gamma_0$ by a factor of $-1$. As this sign is crucial in the computation of $R(0)$ and $R(\infty)$, we choose to avoid ambiguity and instead compute the monodromy matrices directly from the Picard-Fuchs equation.
\end{remark}

Next we consider the general case. Let $\pi_g\colon \calX_g \to \Proj^1$ be a K3 surface fibred Calabi-Yau threefold as in Corollary \ref{XgCY3} and suppose that $g^{-1}(\infty)$ consists of two points \textup{(}so that $l=2$\textup{)}. Recall that $\calX_g$ is defined by a degree $n$ cover $g\colon \mathbb{P}^1 \rightarrow \mathcal{M}_2$ with ramification profiles $[x_1,\dots,x_k]$, $[y_1,\dots,y_l]$ and $[z_1,\dots,z_m]$ over $\lambda = 0$, $\lambda =  \infty$ and $\lambda = \frac{1}{256}$ respectively, and ramification degree $r$ away from these three points. Let $U \subset \Proj^1$ be the preimage $g^{-1}(U_{M_2})$ and let $j\colon U \to \Proj^1$ denote the inclusion. 

Now, by the proof of Proposition \ref{KXgtrivial}, the singular fibres of $\calX_g$ are all either nodal K3 surfaces or normal crossings unions of smooth rational surfaces, so the argument of Example \ref{ex:X2} gives $b_3(\calX_g) = h^1(\Proj^1,j_*\mathcal{T}(\calX_g))$. But, by construction, the local system on $U$ given by $\mathcal{T}(\calX_g)$ is equal to $g^*\mathbb{V}$, where  $\mathbb{V}$ is the local system over $U_{M_2}$ given by $\mathcal{T}(\mathcal{X}_2)$. The cohomology of this local system is computed by:

\begin{proposition} Let $\mathbb{V}$ be the local system over $U_{M_2}$ given by $\mathcal{T}(\mathcal{X}_2)$.
We have 
\[h^1(\mathbb{P}^1, j_* g^*\mathbb{V}) = 2 + 2k + (m_{\mathrm{odd}} - n),\]
where $m_{\mathrm{odd}}$ denotes the number of $z_1,\dots, z_m$ which are odd.

In particular, if $g$ is unramified over $\lambda = \frac{1}{256}$, then $h^1(\mathbb{P}^1, j_* g^*\mathbb{V}) =2 + 2k$.
\end{proposition}
\begin{proof}
If $g$ ramifies to order $a$ at a point $q$ in $\mathbb{P}^1 - U$, then the monodromy of the pulled-back local system $g^*\mathbb{V}$ about $q$ is given by $\gamma_{g(q)}^a$ where $\gamma_{g(q)}$ is the monodromy matrix of $\mathcal{T}(\mathcal{X}_2)$ around $g(q)$. Therefore we can compute, using the explicit expressions for local monodromy found in Example \ref{ex:X2}, that
\begin{itemize}
\item if $g$ ramifies to order $y$ at a preimage of $\infty$, then $R(q) = 4-\hcf(y,4)$,
\item if $g$ ramifies to order $z$ at a preimage of $\frac{1}{256}$, then $R(q) = 2 - \hcf(z,2)$, and
\item at any preimage of $0$, we have $R(q) = 2$.
\end{itemize}
Thus we calculate
\begin{equation}
\label{eq:h21}
h^1(\mathbb{P}^1,j_*g^*\mathbb{V}) = \sum_{i=1}^l (4-\hcf(y_i,4)) + \sum_{j=1}^m (2-\hcf(z_j,2))  + 2k - 6.
\end{equation}

Now we impose the conditions of Proposition \ref{trivialcanonicalprop}. By assumption we have $l=2$, and $y_i = \hcf(y_i,4)$ for both $y_1$ and $y_2$. Equation \eqref{eq:h21} thus gives 
\begin{align*}
h^1(\mathbb{P}^1,j_*g^*\mathbb{V}) &= (4-y_1) + (4 - y_2) + \sum_{j=1}^m(2-\hcf(z_j,2)) + 2k - 6 \\ 
& = 2+ 2k +\left(m_\mathrm{odd} - n\right).
\end{align*}
Note that $(m_\mathrm{odd} - n) \leq 0 $ is an even number.
\end{proof}

Since $\calX_g$ is a Calabi-Yau threefold, we therefore have:

\begin{corollary}\label{h21cor}
Let $\calX_g$ be a Calabi-Yau threefold as in Corollary \ref{XgCY3} and suppose that $g^{-1}(\infty)$ consists of two points \textup{(}so that $l=2$\textup{)}. Then 
\[h^{2,1}(\mathcal{X}_g) = k + \left(\dfrac{m_\mathrm{odd} - n}{2}\right),\]
where $k$ denotes the number of ramification points of $g$ over $\lambda = 0$, $m_{\mathrm{odd}}$ denotes the number of ramification points of odd order of $g$ over $\lambda = \frac{1}{256}$, and $n$ is the degree of $g$.

Moreover, if $g$ is unramified over $\lambda = \frac{1}{256}$, then $h^{2,1}(\mathcal{X}_g) = k = r$, the degree of ramification of $g$ away from $\lambda \in \{0,\frac{1}{256},\infty\}$.
\end{corollary}

\begin{remark} \label{ijrem} We can now explain the general form for the generalized functional invariant maps $g$ of the Calabi-Yau threefolds fibred by $M_2$-polarized K3 surfaces listed in \cite[Theorem 5.10]{flpk3sm} (see Example \ref{ijex}). Indeed, in these cases $h^{2,1}(\calX_g) = 1$ by assumption so, if we assume that the map $g$ is unramified over $\lambda = \frac{1}{256}$ (which guarantees smoothness of $\calX_g$, by Proposition \ref{KXgtrivial}), then $k = r = 1$ by Corollary \ref{h21cor}. From this, we see that $g$ is totally ramified over $\lambda = 0$ and has a single ramification of degree $2$ away from $\lambda \in \{0,\frac{1}{256},\infty\}$. Moreover, if we write $[y_1,y_2] = [i,j]$, for some $i,j \in \{1,2,4\}$, then we must have $\deg(g) = n =  i+j$. 

This shows that, if $g$ is unramified over $\lambda = \frac{1}{256}$ and $h^{2,1}(\calX_g) = 1$, then the generalized functional invariant map $g\colon \Proj^1 \to \calM_{M_2}$ must have the form given in Example \ref{ijex}.
\end{remark}

To conclude this section, we demonstrate the application of this theory by calculating the Hodge numbers in our running example of the quintic mirror threefold:

\begin{example} Recall from Example \ref{ijex} that the fibration of the quintic mirror threefold by $M_2$-polarized K3 surfaces has generalized functional invariant $g$ with  $(k,l,m,n,r) = (1,2,5,5,1)$, $[x_1] = [5]$, $[y_1,y_2] = [1,4]$, and $[z_1,\ldots,z_5] = [1,1,1,1,1]$. The Hodge numbers of this threefold are well known; here we illustrate how to recover them from the results above.

Firstly, we have $h^{2,1}(\calX_g) = k = 1$, by Corollary \ref{h21cor}. Moreover, by Proposition \ref{Xgh11}, we have
\[h^{1,1}(\calX_g) = 20 + (2x_1^2 + 1) + c_1 + c_2 = 20 + 51 + 30 + 0 = 101\]
as expected.
\end{example}

\section{Deformations and Moduli Spaces}\label{defsect}

Now consider the setting where $g$ is unramified over the point $\lambda = \frac{1}{256}$ and has only simple ramification points away from $\lambda \in \{0,\frac{1}{256},\infty\}$. In this case Corollary \ref{h21cor} raises an obvious question. It is easy to see that, in this setting, $r$ is equal to the number of simple ramification points of $g$ away from $\lambda \in \{0,\frac{1}{256},\infty\}$. Moreover, for the corresponding threefolds $\calX_g$, we also have $h^{2,1}(\calX_g) = r$. So to what extent are small deformations of $\calX_g$ determined by the locations of these simple ramification points? 

In more generality, we may ask to what extent small deformations of the threefold $\calX_g$ are determined by deformations of the map $g$. In fact, we find:

\begin{proposition} \label{prop:deffib}
Let $\calX_g$ be a Calabi-Yau threefold as in Corollary \ref{XgCY3} and suppose that $g^{-1}(\infty)$ consists of two points \textup{(}so that $l=2$\textup{)}. Moreover, suppose that $g$ is unramified over  $\lambda = \frac{1}{256}$. Then any small deformation of $\mathcal{X}_g$ is obtained by deforming the map $g$ in a way that preserves the ramification profiles over $\lambda \in \{0,\infty\}$.
\end{proposition}
\begin{proof}
Suppose that $\pi_g\colon \calX_g \to \Proj^1$ and $\pi_{g'}\colon \calX_{g'} \to \Proj^1$ are two Calabi-Yau threefolds defined by maps $g,g' \colon \Proj^1 \to \calM_{M_2}$ satisfying the assumptions of the proposition. We say that an isomorphism $\varphi\colon \mathcal{X}_g \rightarrow \mathcal{X}_{g'}$ is an isomorphism of fibrations between $(\mathcal{X}_g,\pi_g)$ and $(\mathcal{X}_{g'},\pi_{g'})$ if there is a commutative diagram
\[ \xymatrix{
\mathcal{X}_g \ar[rd]_{\pi_g} \ar[rr]^\varphi && \mathcal{X}_{g'} \ar[ld]^{\pi_{g'}}\\
&\mathbb{P}^1 }\]
Such an isomorphism exists if and only if there is an automorphism $\phi\colon \Proj^1 \to \Proj^1$ so that
\[ \xymatrix{
\mathbb{P}^1 \ar[rd]_{g} \ar[rr]^\phi &&\mathbb{P}^1 \ar[ld]^{g'}\\
&{\mathcal{M}}_{{M}_2}  }\]

Assume that $g$ has ramification profile $[y_1,y_2]$ over $\lambda = \infty$ and $[x_1,\ldots,x_k]$ over $\lambda = 0$. By applying an automorphism $\phi$ of $\Proj^1$ as above, we may assume that the ramification points over $\lambda = \infty$ are $(1:0)$ and $(0:1)$, and that $(1:1)$ is a ramification point over $\lambda = 0$ with ramification index $x_1$. Then $g$ may be written as
\[g\colon (s:t) \longmapsto \frac{a_1(s-t)^{x_1}\prod_{i=2}^k (s-a_it)^{x_i}}{s^{y_1}t^{y_2}},\]
with parameters $a_1,\dots,a_k \in (\mathbb{C}^\times)^k - \Delta$ where $\Delta$ is the union of the big diagonals in $(\mathbb{C}^\times)^k$. 

Thus there is an $k$-dimensional space of maps $g\colon \Proj^1 \to \calM_{M_2}$ with the property that $g$ has ramification profile $[y_1,y_2]$ over $\lambda = \infty$ and $[x_1,\ldots,x_k]$ over $\lambda = 0$. By the discussion above, this means that the space of local deformations of the fibration $(\mathcal{X}_g,\pi_g)$ is also $k$-dimensional. 

Now, by a result of Oguiso \cite[Example 2.3]{ffsscy3f}, K3 fibrations on $\mathcal{X}_g$ correspond to certain rational rays in the nef cone of $\mathcal{X}_g$ so, in particular, there are at most countably many  K3 fibrations on $\mathcal{X}_g$. This means that we cannot continuously vary the K3 fibration on $\mathcal{X}_g$ without deforming $\mathcal{X}_g$ itself. Since the K3 fibration $(\mathcal{X}_g,\pi_g)$ may be deformed in $k$ different directions and the deformation space of $\calX_g$ is $k$-dimensional (by Corollary \ref{h21cor}), the claim follows.
\end{proof}

\begin{remark}
From the proof of this proposition the reader may note that, under the assumptions that $g^{-1}(\infty)$ consists of two points and $g$ is unramified over  $\lambda = \frac{1}{256}$, an open subset of the moduli space of K3-fibred Calabi-Yau threefolds $(\mathcal{X}_g,\pi_g)$ is identified with a moduli space of maps $g$ with fixed ramification profiles over $\{0,\frac{1}{256},\infty\}$. Moduli spaces of such maps are called \emph{Hurwitz spaces} and have been studied extensively in the literature.
\end{remark}

Next, we will show that any Calabi-Yau threefold $\calX_g$ is deformation equivalent to a Calabi-Yau threefold $\calX_{g'}$ defined by a map $g'\colon \Proj^1 \to \calM_{M_2}$ with only simple ramification away from $\lambda \in \{0,\frac{1}{256},\infty\}$. In particular this shows that, if we are only interested in generic members of deformation classes, we can safely ignore the type of ramification away from  $\lambda \in \{0,\frac{1}{256},\infty\}$.

\begin{proposition}
\label{prop:simple}
Let $g\colon \Proj^1 \to \calM_{M_2}$ be a morphism. Then there exists a deformation $g'\colon \Proj^1 \to \calM_{M_2}$ of $g$ that has the same ramification profiles as $g$ over $\lambda \in \{0,\frac{1}{256},\infty\}$ and only simple ramification away from these points.

Thus, if $\calX_g$ is a Calabi-Yau threefold as in Corollary \ref{XgCY3}, then $\calX_g$ is deformation equivalent to a Calabi-Yau threefold $\calX_{g'}$ defined by a map $g'\colon \Proj^1 \to \calM_{M_2}$ that is simply ramified away from $\lambda \in \{0,\frac{1}{256},\infty\}$.
\end{proposition}

\begin{remark} We note that this result is not unexpected: neither our computation of $h^{1,1}(\mathcal{X}_g)$ nor our computation of $h^{2,1}(\mathcal{X}_g)$ made any reference to the type of ramification away from $\lambda \in \{0,\frac{1}{256},\infty\}$, so we should not expect such ramification to affect the deformation type of $\calX_g$. 
\end{remark}

\begin{proof} Assume that $g$ has degree $n$ and let $\Sigma = \{p_1,\dots,p_s\}$ be the set of branch points of $g$ in $\calM_{M_2}$. Choose a set of discs $\Delta_i$ around each $p_i \in \Sigma$, small enough that no pair of discs intersects, and choose non-intersecting paths $\beta_i$ from a base-point $p \in \Proj^1$ to the boundary of each $\Delta_i$. For each $\Delta_i$, let $\gamma_i$ be the path obtained by following the path $\beta_i$ from $p$ to the boundary of $\Delta_i$, going around $\partial \Delta_i$ once counterclockwise, then traversing $\beta_i$ backwards to $p$. 

The classes of $\gamma_i$ generate $\pi_1(p\,,\,\calM_{M_2} - \Sigma)$ and the concatenation $\gamma_1 \cdots  \gamma_s$ is a contractible loop. Label the points above $p$ by the integers $\{1,\ldots,n\}$. Then to each point $p_i \in \Sigma$, we may associate an element $\sigma_i \in S_n$ which describes the action of monodromy around $\gamma_i$ on the points over $p$. This monodromy representation determines $g$ up to reordering of the points over $p$. Since $\Proj^1$ is connected, the subgroup of $S_n$ generated by $\{\sigma_1,\ldots,\sigma_s\}$ is transitive.

If $g$ is not simply ramified away from $\lambda \in \{0,\frac{1}{256}, \infty\}$, then there exists a $p_i \notin \{0,\frac{1}{256},\infty\}$ so that the corresponding $\sigma_{i}$ is not a transposition. Let $\sigma_{i} = \tau_1 \cdots \tau_{s'}$ be a minimal decomposition of such a $\sigma_{i}$ into transpositions. Then we claim that the set $P' = \{\sigma_1 ,\dots, \sigma_{{i-1}}, \tau_1,\dots, \tau_{s'},\sigma_{{i+1}},\dots, \sigma_{s}\}$ can be used to define a new cover $g'\colon \Proj^1 \to \calM_{M_2}$, so that the points $q_j$ over which $g'$ ramifies with cycle structure $\tau_j$ have only simple ramification. 

To define this cover, let $t$ be a complex coordinate on the disc $\Delta_i$, chosen so that $\Delta_i = \{t \in \C \mid |t|<1\}$ and $p_i$ lies at $t = 0$. Let $q$ denote the point where the path $\beta_i$ meets the boundary of $\Delta_i$. Take points $q_1,\dots , q_{s'}\in \Delta_i$ and define non-intersecting loops $\delta_j$ from $q$ around each $q_j$. Then let $\gamma'_j$ denote the path obtained by following the path $\beta_i$ from $p$ to $q$, going around $\delta_j$ once counterclockwise, then traversing $\beta_i$ backwards to $p$. After relabelling if necessary, we may assume that $\gamma_i = \gamma'_1\cdots\gamma'_{s'}$ in $\pi_1(p\,,\,\calM_{M_2} - \Sigma')$, where $\Sigma' := \{p_1,\dots, p_{i-1},q_1,\dots, q_{s'}, p_{i+1},\dots p_s\}$.

By construction, the set $\{\gamma_1,\ldots,\gamma_{i-1},\gamma'_1,\ldots,\gamma_{s'}',\gamma_{i+1},\ldots,\gamma_s\}$ forms a basis for $\pi_1(p\,,\,\calM_{M_2} - \Sigma')$. Assign the set $P'$ of elements of $S_n$ to these loops by associating $\sigma_i$ to $\gamma_i$ and $\tau_j$ to $\gamma'_j$. This defines a representation  $\rho\colon\pi_1(p\,,\,\calM_{M_2}- \Sigma') \to S_n$ whose image is transitive by construction so, by the Riemann existence theorem, there is a unique connected curve $C$ and morphism $g'\colon C \rightarrow\calM_{M_2}$, such that $g'$ is branched over $\Sigma'$ and $\rho$ is the monodromy representation of the associated Galois cover of $\calM_{M_2}- \Sigma'$. Using the Riemann-Hurwitz formula, it is easy to check that $C \cong \Proj^1$

Now take a deformation $g_t'$ of $g'$, obtained by multiplying $q_1,\ldots, q_{s'}$ by the local coordinate $t$, and an appropriate deformation of the loops $\gamma_j'$. At $t=0$, the points $q_j$ all go to $p_i$ and the map $g$ degenerates to a map $g'_0$ whose monodromy about $p_i$ is $\tau_1\cdots \tau_{s'} = \sigma_i$. By the uniqueness part of the Riemann existence theorem, the map $g'_0$ is exactly $g$. 

We may now repeat this procedure for each $p_i \neq \{0,\frac{1}{256},\infty\}$ over which the corresponding ramification of $g$ is not simple, to obtain a map $g'\colon \Proj^1 \to \calM_{M_2}$ that deforms to $g$ and has simple ramification away from $\{0,\frac{1}{256},\infty\}$.

Given this, the statement about the threefolds $\calX_g$ and $\calX_{g'}$ follows from the fact that the deformation $g \leadsto g'$ induces a deformation $\calX_g \leadsto \calX_{g'}$. As this deformation does not affect a neighbourhood of the fibres above $\lambda \in \{0,\frac{1}{256},\infty\}$, we see that $\calX_{g'}$ must also be Calabi-Yau.
\end{proof}

We conclude this section by asking what happens to the threefolds $\calX_g$ when the map $g$ degenerates. Such degenerations occur when a ramification point away from $\lambda \in \{0,\frac{1}{256},\infty\}$ moves to one of these points. In this situation it is easy to see what occurs: the ramification profile defining $\calX_g$ changes and the threefold becomes singular. If the new ramification profile defines a smooth Calabi-Yau (according to Corollary \ref{XgCY3}), then this singular threefold admits a Calabi-Yau resolution, with new Hodge numbers given by Proposition \ref{Xgh11} and Corollary \ref{h21cor}. Geometrically, the Calabi-Yau threefold $\calX_g$ undergoes a \emph{geometric transition} to a new Calabi-Yau threefold with different Hodge numbers.

\begin{example} \label{moduliex} In our running example of the quintic mirror threefold, the generalized functional invariant $g$ has one simple ramification away from $\lambda \in \{0,\frac{1}{256},\infty\}$. As noted above, moving this ramification point corresponds to deforming the quintic mirror in its ($1$-dimensional) complex moduli space. This can also be seen from the explicit form of the generalized functional invariant given in Example \ref{ijex}, where varying the modular parameter $A$ changes the location of this simple ramification point, whilst keeping the other ramification points fixed.

It is easy to see what happens when this simple ramification point moves to $\lambda \in \{0,\frac{1}{256},\infty\}$. At $\lambda = \frac{1}{256}$ (corresponding to $A = \frac{1}{5^5}$), the proof of Proposition \ref{KXgtrivial} shows that $\calX_g$ acquires a single isolated $cA_1$ (node) singularity. Moreover, by \cite[Lemma 3.5]{gscyt} and \cite[Theorem 2.5]{slmcyt}, the resulting singular threefold is $\Q$-factorial, so does not admit a crepant resolution. In particular, this degeneration provides an example of a map $g$ that satisfies the conditions of Proposition \ref{KXgtrivial} but does \emph{not} give rise to a smooth Calabi-Yau threefold.

When the simple ramification point moves to $\lambda = \infty$ (corresponding to $A = \infty$), the map $g$ becomes totally ramified over $\lambda \in \{0,\infty\}$. The degenerate threefold acquires an additional $\Z/5\Z$ action, which acts to permute the sheets of this cover.

Finally, when the simple ramification point moves to $\lambda = 0$ (corresponding to $A = 0$), we obtain a degeneration with \emph{maximally unipotent monodromy} (see \cite[Chapter 6]{msag}).

As we can see, we have obtained the three well-known boundary points in the complex moduli space of the quintic mirror threefold.
\end{example}

\bibliography{Books}
\bibliographystyle{amsalpha}
\end{document}